\UseRawInputEncoding 
\documentclass[12pt,fceqn]{article}
\usepackage{mathrsfs}
\pdfoutput=1

\usepackage{bbm}
\usepackage{amsfonts}
\usepackage{amssymb,graphics,graphicx,subfigure,caption2,rotating}

\usepackage{amsmath, latexsym}
\usepackage[amsmath,thmmarks]{ntheorem}

\newtheorem{theorem}{Theorem}[section]

\newtheorem{definition}[theorem]{Definition}

\newenvironment{proof}{\noindent{\bf Proof.} \hskip .2truein \rm}

\begin{document}

\begin{center}
{\Large\bf Fuzzy Riesz homomorphism on fuzzy Riesz space}

\footnote{\hskip 0.25cm E-mail addresses: chengnaanna@126.com}

\footnote{\hskip 0.25cm The paper was supported by the research fund of
the National Natural Science Foundation of China (Grant
No.11801454) and the Spring Plan of the Ministry of Education of China in 2018. }

\normalsize{\it \rm Na Cheng $^{1}$, Guanggui Chen $^{2}$}

 \footnotesize{\it \rm ~$(1)$$(1)$School of Science, Xihua University;}

\footnotesize{\it \rm ~ Chengdu Sichuan 610031, P. R. China }
\end{center}

\vspace{0.5cm}

\small\begin{quote} {\bf Abstract:}\rm In this paper, we give some properties of fuzzy Riesz homomorphism on fuzzy Riesz space. We give definitions of fuzzy quotient spaces, and some characterizations of fuzzy Archimedean quotient spaces. We prove the properties which fuzzy Riesz homomorphism is fuzzy order continuous, and study the extension of fuzzy lattice homomorphism and the extension of fuzzy order continuous lattice homomorphism. Finally, we investigate the factorization by fuzzy Riesz homomorphisms.

\noindent{\bf Key words:} \rm fuzzy Riesz space, fuzzy Riesz homomorphism, fuzzy Dedekind complete, fuzzy quotient space, fuzzy order continuous.

\noindent{\bf MSC:} \rm 46S40; 03E70; 06D72
\end{quote}

\normalsize

\section{Introduction}
\ \ \ \
The concept of fuzzy set was initiated by Zadeh [1]. Later, in 1971 [2], he introduced the notion of fuzzy relation, fuzzy order was initiated by generalizing the notions of reflexivity, antisymmetry and transitivity. Then Venugopalan [3] defined and studied the fuzzy ordered sets. These led Beg and Islam to studies on fuzzy Riesz spaces in [4], fuzzy ordered linear spaces in [5], $\sigma-$complete fuzzy Riesz spaces in [6] and fuzzy Archimedean spaces in [7]. Bag and Samanta []studied some properties of finite dimensional fuzzy norm spaces, and introduced another type of bound of linear operators and fuzzy dual spaces, proved analogues of fundamental theorems in functional analysis. Hong [8] introduced the concepts of fuzzy Riesz subspaces, fuzzy ideal, fuzzy bands and fuzzy band projections. Keun Young Lee [30] provided characterizations of approximation properties in fuzzy normed spaces.

Order-homomorphisms plays an important role in fuzzy topology. Based on the concept of order-homomorphism on fuzzy lattices, study on such notions reveals some of the further connections between two different current approaches to fuzzy topology[15-19].

Grzymala-Busse introduced the homomorphism between information systems to study transformations of information systems while preserving some data structures. It is an effective tool to study relationships between information systems. It is useful for aggregating sets of objects, attributes, and descriptors of the original system. This is a good approach for reducing data volume of information systems[21-26].

Homomorphism are very important in applications, they play a crucial role in topology, algebras, vector valued functions, information systems. Hence, the fuzzy type of them can also play an important part in the new fuzzy areas. We study the fuzzy Riesz homomorphism in terms of crips counter part which is very important for the development of fuzzy Riesz space.

Mobashir Iqbal and Zia Bashir in [9] proposed the notion of fuzzy Riesz homomorphism for the use of the existence of fuzzy Dedekind completion of Archimedean fuzzy Riesz space. They gave some proposition of fuzzy Riesz homomorphism that used in the fuzzy Dedekind completion. We are interested in more properties of fuzzy Riesz homomorphism. In this paper, Section 2 is devoted to some properties of fuzzy Riesz homomorphism on fuzzy Riesz space. In Section 3, we give definitions of Fuzzy quotient spaces, and some characterizations of fuzzy Archimedean quotient spaces. In Section 4, we prove the properties which fuzzy Riesz homomorphism is fuzzy order continuous, and study the extension of  fuzzy lattice homomorphism and the extension of  fuzzy order continuous lattice homomorphism. In Section 5, we  investigate the factorization by fuzzy Riesz homomorphisms.

Throughout the paper, we use the following notation.

\begin{definition}Let $E$ be a crisp set. A fuzzy order on $E$ is a fuzzy subset of $E \times E$ such that the following conditions are satisfied:

(i) (reflexivity)for all $x \in E, \mu(x, x)=1$

(ii)(antisymmetry) for $x, y \in E, \mu(x, y)+\mu(y, x)>1$ implies $x=y$

(iii)(transitivity)for $x, z \in E$$$\mu(x, z) \geq \bigvee_{y \in E}[\mu(x, y) \wedge \mu(y, z)]$$

where $\mu: E \times E \rightarrow[0,1]$ is the membership function of the fuzzy subset of $E \times E$. A set with a fuzzy order defined on it is called a fuzzy ordered set (foset, for short).
\end{definition}

Let $E$ be a foset and $x \in E,$ $\downarrow x$ denotes the fuzzy set on $E$ defined by $(\downarrow x)(y)=\mu(y, x)$ for all $y \in E .$ $\uparrow x$ denotes the fuzzy set on $E$ defined by $(\uparrow x)(y)=\mu(x, y)$ for all $y \in E .$ If $A$ is a crisp subset of $E, $ then $\uparrow A=\underset{x \in A}{\bigcup}(\uparrow x)$ and $\downarrow A=\underset{x \in A}{\bigcup}(\downarrow x)$.

\begin{definition}Let $A$ be a crisp subset of a foset $E$, then the upper bound $U(A)$ of $A$ is the fuzzy set on $E$ defined as follows:

$$
U(A)(y)=\left\{\begin{array}{ll}
0 & \text { if }(\uparrow x)(y) \leq \frac{1}{2} \text { for some } x \in A \\
\left(\underset{x \in A}{\bigcap} \uparrow x\right)(y) & \text { otherwise }
\end{array}\right.
$$
The lower bound $L(A)$ of $A$ is the fuzzy set on $E$ defined by:
$$
L(A)(y)=\left\{\begin{array}{ll}
0 & \text { if }(\downarrow x)(y) \leq \frac{1}{2} \text { for some } x \in A \\
\left(\underset{x \in A}{\bigcap} \downarrow x\right)(y) & \text { otherwise }
\end{array}\right.
$$

\end{definition}

When $U(A)(x)>0,$ for some $x\in E$, we write $x \in U(A).$ In this case, we say $A$ is bounded from above and the element $x$ is called an upper bound of $A$. A subset $A$ of $E$ is said to be bounded from below if there exists an element $x \in E$ such that $x \in L(A) .$ Such element $x$ is called a lower bound of $A$. A subset $A$ of $E$ is bounded if it is bounded from above and from below.

The element $z$ is called the supremum of $A$ (written as $z=\sup A$ ), if (i) $z \in U(A)$
(ii) if $y \in U(A)$ implies $y \in U(z)$. The element $z$ is called the infimum of $A$ (written as $z=\inf A),$ if $(\mathrm{i}) z \in L(A)$ $(\mathrm{ii})$ if $y \in L(A),$ implies $y \in L(z)$.

\begin{definition}A (real) linear space $E$ is said to be a fuzzy ordered linear space if $E$ is a foset satisfying the following conditions:

(i) if $x_{1}, x_{2} \in E$ such that $\mu\left(x_{1}, x_{2}\right)>1 / 2$ then $\mu\left(x_{1}, x_{2}\right) \leq \mu\left(x_{1}+x, x_{2}+x\right)$ for all $x \in E$;

(ii) if $x_{1}, x_{2} \in E$ such that $\mu\left(x_{1}, x_{2}\right)>1 / 2$ then $\mu\left(x_{1}, x_{2}\right) \leq \mu\left(\alpha x_{1}, \alpha x_{2}\right)$ for all
$0<\alpha \in \mathbb{R}$.

\end{definition}

\begin{definition}A fuzzy ordered vector space is called a fuzzy Riesz space if it is also a fuzzy lattice at the same time.
\end{definition}

\begin{definition}Let $E$ be a fuzzy Riesz spaces, a vector subspace $K$ of $E$ is said to be a fuzzy Riesz subspace if for all $x,y\in K$, the elements $x\vee y$ and $x\wedge y$ belong to $K$.
\end{definition}

\begin{definition}A subset $A$ of $E$ is said to be fuzzy solid if it follows from $u(|x|,|y|)>{\frac{1}{2}}$ and $y\in A$ that $x\in A$. In this case, we say $A$ is a fuzzy solid subset of $E$. A fuzzy solid vector subspace $I$ of $E$ is called a fuzzy ideal of $E$.
\end{definition}

\begin{definition}Let $D$ be a subset of a fuzzy Riesz space $E$. The smallest fuzzy ideal of $E$ that contains $D$ is called the fuzzy ideal generated by $D$ and is denoted by $I_{D}$. If $D$ is a singleton, that is, $D=\{x\}$ for some $x\in E$, then $I_{D}$ is written as $I_{x}$ and is called the principal fuzzy ideal generated by $x$.
\end{definition}

\begin{definition}Let $(E,u)$ be a fuzzy Riesz space, $E$ is said to be fuzzy Dedekind complete if every nonempty subset of $E$ that is bounded above has a supremum.
\end{definition}

\begin{definition}A directed ordered fuzzy ordered vector space $E$ is said to be a fuzzy Archimedean space if the set $\{\lambda x| \lambda>0\}$ is not bounded above for any nonnegative element $x\in E$. In this case, we also say the space $E$ is fuzzy Archimedean.
\end{definition}

\begin{definition}Let $(E,u)$, $(F,v)$ be two fuzzy Riesz spaces, an operator $T:E\rightarrow F$ is said to be fuzzy positive if $u(0,x)>{\frac{1}{2}}$ implies $v(0,T(x))>{\frac{1}{2}}$.
\end{definition}

\begin{definition}Let $(E,u)$,$(F,v)$ be two fuzzy Riesz spaces, $T:E\rightarrow F$ be a fuzzy positive operator. $T$ is said to be fuzzy order bounded if $T(C)\subseteq F$ is fuzzy ordered bounded whenever $C\subseteq E$ is fuzzy order bounded.
\end{definition}

\begin{definition}Let $(E,u)$,$(F,v)$ be two fuzzy Riesz spaces, $FL(E,F)$ denotes the set of all fuzzy linear operators between $(E,u)$ and $(F,v)$. $FL_{b}(E,F)$ denotes the set of all fuzzy order bounded operators between $(E,u)$ and $(F,v)$.
\end{definition}

Note that $FL(E,F)$ is a vector space with pointwise operations. $FL(E,F)$ under the fuzzy ordering $S\leq T$ whenever $T-S$ is a fuzzy positive operator is a fuzzy ordered vector space.
\begin{definition}Let $(E,u)$,$(F,v)$ be two fuzzy Riesz spaces, a function $p:E\rightarrow F$ is called fuzzy sublinear whenever

(1)$v(p(x+y),p(x)+p(y))>{\frac{1}{2}}$ holds for all $x,y\in E$;

(2)$p(\lambda x)=\lambda p(x)$ holds for all $x\in E$ and all $\lambda\in R^{+}$.
\end{definition}

\begin{definition}Let $E$ be fuzzy Riesz spaceㄛthe fuzzy ideal $A$ is called a fuzzy $\sigma-$ideal in $E$ whenever ${x_{n}}\subseteq (A)^{+}$ and $x_{n}\uparrow x$ in $E$ imply $x\in A$.
\end{definition}

We refer to [4] - [8] for any unexplained terms from the theory of
fuzzy Riesz space.
\section{Properties of fuzzy Riesz homomorphism}
\begin{definition}Let $(E,\mu)$, $(F,\nu)$ be two fuzzy Riesz spaces, an operator $T:E\rightarrow F$ is said to be fuzzy Riesz homomorphism if $T(x\vee y)=T(x)\vee T(y)$ holds for all $x,y\in E$. In addition, If $T$ is bijective, then it is said to be fuzzy Riesz isomorphism.
\end{definition}

The operator $T:C[0,1]\rightarrow L_{1}[0,1]$ by $Tf(t)=tf(t)$, $t\in[0,1]$, is a fuzzy lattice homomorphism. The operator $T:C[0,1]\rightarrow R$ by $T(f)=\int_0^1 f(t)d t$, is not a fuzzy lattice homomorphism.

\begin{theorem}Let $(E,u)$ and $(F,v)$ be two fuzzy Riesz spaces, $T:E\rightarrow F$ be a fuzzy Riesz homomorphism, then the following statements holds:

(1)$v(0,Tx)>{\frac{1}{2}}$ if and only if there exists an element $u(0,z)>{\frac{1}{2}}$ in the $Ker(T)$ such that $u(0,x+z)>{\frac{1}{2}}$. Then $v(Tx,Ty)>{\frac{1}{2}}$ if and only if there exists an element $w\in E$ such that $w\geq x$, $w\geq y$ and $Tw=Tx$.

(2)$v(|Tx|,|Ty|)>{\frac{1}{2}}$ if and only if there exists an element $w\in Ker(T)$ such that $u(|w|,|x|)>{\frac{1}{2}}$ and $u(|x-w|,|y|)>{\frac{1}{2}}$.
\end{theorem}

\begin{proof}(1) Let $v(0,Tx)>{\frac{1}{2}}$, Theorem 3.16 of [9] implies $Tx=T(x^{+})=(Tx)^{+}$, it follows that $T(x^{+}-x)=0$, hence $T(x^{-})=0$. Put $z=x^{-}$, we have $u(0,z)>{\frac{1}{2}}$ and $z\in Ker(T)$, as $x+z=x+x^{-}=x^{+}$, showing that $u(0,z)>{\frac{1}{2}}$. On the other hand, let $z\in Ker(T)$ and $u(0,x+z)>{\frac{1}{2}}$, since $Tx=T(x+z)$, we have $v(0,Tx)>{\frac{1}{2}}$. This shows that $v(Ty,Tx)>{\frac{1}{2}}$ if and only if there exists an element $u(0,z)>{\frac{1}{2}}$ and $z\in Ker(T)$ such that $u(0,x-y+z)>{\frac{1}{2}}$. Take $w=x+z$, then $v(Ty,Tx)>{\frac{1}{2}}$ if and only if there exists an element $w\in E$ such that $u(x,w)>{\frac{1}{2}}$, $u(y,w)>{\frac{1}{2}}$ and $Tx=Tw$.

(2)Take $v(|Tx|,|Ty|)>{\frac{1}{2}}$, Theorem 3.16 of [9] implies $v(T|x|,T|y|)>{\frac{1}{2}}$, properties (1) implies that there exists an element $z\in Ker(T)$ such that $u(|x|,|y|+z)>{\frac{1}{2}}$.
Since $u((|x|-|y|),z)>{\frac{1}{2}}$ and $u((|x|-|y|),|x|)>{\frac{1}{2}}$, it follows that $u((|x|-|y|),|x|\wedge z)>{\frac{1}{2}}$. Thus, $u(|x|,|y|+|x|\wedge z)>{\frac{1}{2}}$ with $|x|\wedge z\in Ker(T)$. This shows we may assume that $u(|x|,|y|+ z)>{\frac{1}{2}}$ with $z\in Ker(T)$ and $u(0,z)>{\frac{1}{2}}$, $u(z,|x|)>{\frac{1}{2}}$. Theorem 4.12 of [4] implies that there exist elements $z_{1},z_{2}\in E^{+}$ such that $z=z_{1}+z_{2}$ and $u(z_{1},x^{+})>{\frac{1}{2}}$, $u(z_{2},x^{-})>{\frac{1}{2}}$. Then, Proposition 4.7 of [4] implies $|x|-z=(x^{+}-z_{1})+(x^{-}-z_{2})=|(x^{+}-z_{1})+(x^{-}-z_{2})|=|(x^{+}-z_{1})-(x^{-}-z_{2})|$. Take $w=z_{1}-z_{2}$, it follows that $|x|-z=|x-w|$. As $u((|x|-z),|y|)>{\frac{1}{2}}$, it follows that $u(|x-w|,|y|)>{\frac{1}{2}}$ with $w\in Ker(T)$. Furthermore, we have $|w|=|z_{1}-z_{2}|=z_{1}+z_{2}=z$ and $u(|w|,|x|)>{\frac{1}{2}}$. Conversely, take $w\in Ker(T)$ with $u(|x-w|,|y|)>{\frac{1}{2}}$. Since $u(||x-w|-|x||,|w|)>{\frac{1}{2}}$, it follows that $(|x-w|-|x|)\in Ker(T)$. Hence, $T(|x|)=T(|x-w|)$. By the hypothesis $u(|x-w|,|y|)>{\frac{1}{2}}$, we know $v(T(|x|),T(|y|))>{\frac{1}{2}}$, this shows that $v(|Tx|,|Ty|)>{\frac{1}{2}}$.\ \ \ \ $\Box$
\end{proof}

\begin{theorem}Let $E$ and $F$ be fuzzy Riesz spaces, $T:E\rightarrow F$ be a fuzzy Riesz homomorphism.

(1)If $Z$ is a fuzzy Riesz subspace of $E$, then the image $T(Z)$ is fuzzy Riesz subspace of $F$;

(2)If $W$ is a fuzzy Riesz subspace of $F$, the image $T^{-1}(w)=(x:x\in Z,Tx\in W)$ is a fuzzy Riesz subspace of $E$.
\end{theorem}

\begin{proof}Suppose that $Z$ is a fuzzy Riesz subspace of $E$, it suffices to show that $T(Z)$ is a fuzzy Riesz subspace of $F$. To this end, let $h,q\in T(Z)$, there exists $x,y\in Z$ with $Tx=h$, $Ty=q$. The definition of fuzzy Riesz homomorphism and the fact $Z$ is a fuzzy Riesz subspace of $E$ imply that $Tx\vee Ty=T(x\vee y)\in T(Z)$. Hence, $h\vee q\in T(Z)$. Therefore, $T(Z)$ is a fuzzy Riesz subspace of $F$.

Next, suppose that $W$ is the fuzzy Riesz subspace of $F$, it suffices to show that $T^{-1}(w)$ is the fuzzy Riesz subspace of $E$. To this end, let $x,y\in T^{-1}(w)$, thus $Tx,Ty\in W$. Since $W$ is the fuzzy Riesz subspace of $F$ and definition of fuzzy Riesz homomorphism imply that $T(x\vee y)=(Tx)\vee (Ty)\in W$. That is $x\vee y\in T^{-1}(w)$. This proves that $T^{-1}(w)$ is a fuzzy Riesz subspace of $E$.\ \ \ \ $\Box$
\end{proof}

\begin{theorem}Let $(E,u)$ and $(F,v)$ be fuzzy Riesz spaces, $T:E\rightarrow F$ be a fuzzy Riesz homomorphism.

(1)If $B$ is a fuzzy ideal in $E$, then $T(B)$ is a fuzzy ideal in $T(E)$.

(2)If $B_{1}$ and $B_{2}$ are two fuzzy ideal in $E$, then $T(B_{1}\cap B_{2})=T(B_{1})\cap T(B_{2})$.
\end{theorem}

\begin{proof}Let $x\in B^{+}$, then $Tx\in T(B)$. Suppose $v(z,Tx)>{\frac{1}{2}}$ with $z\in (T(E))^{+}$, we need to show that $z\in T(B)$. As $z\in T(E)$, we have $z=Ty$ for some $y\in E$. In view of $z\in (T(E))^{+}$, we have $z=z^{+}=(Ty)^{+}=T(y^{+})$. Take $w=x\wedge y^{+}$, then $w\in B$. In view of that $T$ is fuzzy Riesz homomorphism and $v(Ty^{+},Tx)>{\frac{1}{2}}$, we have $Tw=T(x\wedge y^{+})=Tx\wedge Ty^{+}=Ty^{+}=z$. Thus, there exists $w\in B$ such that $Tw=z$, this shows that $z\in T(B)$. Therefore, $T(B)$ is a fuzzy ideal in $T(E)$.

Next, let $B_{1}$ and $B_{2}$ be two fuzzy ideals in $E$. First, we need to show that $T(B_{1}\cap B_{2})\subseteq T(B_{1})\cap T(B_{2})$. Take $f\in T(B_{1}\cap B_{2})$, there exists $x\in B_{1}\cap B_{2}$ such that $T(x)=f$. Note that if $x\in B_{1}\cap B_{2}$, then $T(x)\in T(B_{1})$ and $T(x)\in T(B_{2})$ this implies that $T(x)\in T(B_{1})\cap T(B_{2})$. Consequently, $f\in T(B_{1})\cap T(B_{2})$. On the other hand, pick $f\in T(B_{1})\cap T(B_{2})$, it follows that $|f|\in T(B_{1})\cap T(B_{2})$. Then there exists $x\in B^{+}_{1}$, $y\in B^{+}_{2}$ such that $|f|=Tx$ and $|f|=Ty$. Put $w=x\wedge y$, then $w\in B_{1}\cap B_{2}$, note that $T$ is a fuzzy Riesz homomorphism, we have $Tw=T(x\wedge y)=Tx\wedge Ty=|f|$. Thus there exists $w\in B_{1}\cap B_{2}$ such that $Tw=|f|\in T(B_{1}\cap B_{2})$, as claimed.\ \ \ \ $\Box$
\end{proof}

\begin{theorem}Let $(E,u)$ and $(F,v)$ be fuzzy Riesz spaces, $T:E\rightarrow F$ be a fuzzy Riesz homomorphism.

(1)If $B$ is a fuzzy ideal in $T(E)$, then $T^{-1}(B)$ is a fuzzy ideal in $E$.

(2)If $B_{1}$ is a fuzzy ideal in $F$, then $T^{-1}(B_{1})$ is a fuzzy ideal in $E$.
\end{theorem}

\begin{proof}Let $B$ be the fuzzy ideal in $T(E)$, suppose that $u(x,y)>{\frac{1}{2}}$ with $y\in T^{-1}(B)$ and $x\geq0$. We need to show that $x\in T^{-1}(B)$. As $y\in T^{-1}(B)$, it follows that $Ty\in B$. Since $T$ is a fuzzy Riesz homomorphism, it follows that $v(Tx,Ty)>{\frac{1}{2}}$ and $v(0,Tx)>{\frac{1}{2}}$. Now the fact that $B$ is the fuzzy ideal in $T(E)$ implies $Tx\in B$. This shows that $x\in T^{-1}(B)$.

Next, if $B_{1}$ is the fuzzy ideal in $F$, then $B_{1}\cap T(E)$ is a fuzzy ideal in $T(E)$. Thus $T^{-1}(B_{1})=T^{-1}(B_{1}\cap T(E))$ is the fuzzy ideal in $E$.\ \ \ \ $\Box$
\end{proof}

\begin{theorem}Let $(E,u)$ and $(F,v)$ be fuzzy Riesz spaces, $T:E\rightarrow F$ be a fuzzy Riesz homomorphism. If $I_{z}$ is the fuzzy principal ideal in $E$ generated by an element $z\in E^{+}$, then $T(I_{z})$ is the fuzzy principal ideal generated in $T(E)$ by the element $Tz$.
\end{theorem}
\begin{proof}Since $I_{z}$ is the fuzzy ideal in $E$, it follows that $T(I_{z})$ is the fuzzy ideal in $T(E)$. It is sufficient to prove that $T(I_{z})$ is the fuzzy principal ideal generated in $T(E)$ by the element $Tz$. In view of that $z\in E^{+}$, $z\in I_{z}$, it follows that $Tz\in T(I_{z})$, hence the fuzzy ideal generated by the element $Tz$ in $T(E)$ is a subset of $T(I_{z})$. On the other hand, let $m\in (T(I_{z}))^{+}$, there exists $x\in I_{z}$ satisfying $m=Tx$. The definition 5.2 of [8] implies that there exists some real $\alpha>0$ with $u(x,\alpha z)>{\frac{1}{2}}$. Thus, $v(m,\alpha Tz)>{\frac{1}{2}}$. Therefore, $m$ is a member of the fuzzy ideal generated  by the element $Tz$ in $T(E)$. This shows that $T(I_{z})$ is included in the fuzzy ideal generated by $Tz$.\ \ \ \ $\Box$
\end{proof}

\begin{theorem}Let $E$ and $F$ be fuzzy Riesz spaces, $T:E\rightarrow F$ be a fuzzy Riesz homomorphism. If $B$ is the subset of $E$, then $T(B^{d})$ is the subset of $(T(B))^{d}$.
\end{theorem}
\begin{proof}Suppose $z\in T(B^{d})$, then there exists $x\in B^{d}$ satisfies $Tx=z$. Take $y\in B$, so $x\perp y$ and $Ty\in T(B)$. We have $|Tx|\wedge |Ty|=T|x|\wedge T|y|=T(|x|\wedge|y|)=0$, showing that $Tx\perp Ty$. Thus, $z=Tx\in (T(B))^{d}$.\ \ \ \ $\Box$
\end{proof}

\begin{theorem}Let $E$ and $F$ be fuzzy Riesz spaces, $T:E\rightarrow F$ be a fuzzy Riesz homomorphism. Then the image of a fuzzy projection band in $E$ is a fuzzy projection band in $T(E)$.

\end{theorem}

\begin{proof}Let $E_{1}$ be a fuzzy projection band in $E$, and $E_{2}$ the disjoint complement of $E_{1}$, so $E=E_{1}\oplus E_{2}$. Theorem 2.3 implies that the images $T(E_{1})$, $T(E_{2})$ are the fuzzy ideals in $T(E)$, and we have $T(E_{1})\bot T(E_{2})$. Take $f\in T(E)$, we have $f=Tx$ for some $x\in E$. Put $x=x_{1}+x_{2}$ with $x_{1}\in E_{1}$, $x_{2}\in E_{2}$, it follows that $f=Tx=Tx_{1}+Tx_{2}$ with $Tx_{1}\in T(E_{1})$ and $Tx_{2}\in T(E_{2})$. This shows $T(E)=T(E_{1})+T(E_{2})$, hence $T(E_{1})$ and $T(E_{2})$ are fuzzy bands in $T(E)$. The definition of the fuzzy projection band shows that $T(E_{2})$ are fuzzy projection bands in $T(E)$.\ \ \ \ $\Box$
\end{proof}

\begin{definition}Let $E$ and $F$ be fuzzy Riesz spaces, $T:E\rightarrow F$ be a fuzzy Riesz homomorphism, $T$ is said to be a fuzzy Riesz $\sigma-$homomorphism whenever $x=\sup x_{n}$ $(n=1,2,\cdots)$ in $E$ implies $Tx=\sup Tx_{n}$ in $F$.
\end{definition}

\begin{theorem}For a fuzzy Riesz homomorphism $T$ of $(E,u)$ onto $(F,v)$, the following conditions are equivalent.

(1)$T$ is a fuzzy Riesz $\sigma-$homomorphism.

(2)For any fuzzy $\sigma-$ideal $A$ in $F$, the inverse image $T^{-1}(A)$ is a fuzzy $\sigma-$ideal in $E$.

(3)The kernel $Ker T$ of $T$ is a fuzzy $\sigma-$ideal in $E$.
\end{theorem}

\begin{proof}$(1)\Rightarrow(2)$Let $T$ be a fuzzy Riesz $\sigma-$homomorphism and $A$ a fuzzy $\sigma-$ideal in $F$. Take $y_{n}\uparrow y$ with $y_{n}\in (T^{-1}(A))^{+}$. Since $T$ is a fuzzy Riesz $\sigma-$homomorphism, it follows that $Ty_{n}\in A^{+}$ with $Ty_{n}\uparrow Ty$. Note that $A$ is a fuzzy $\sigma-$ideal, we have $Ty\in A$, hence $y\in T^{-1}(A)$. This shows that $T^{-1}(A)$ is a fuzzy $\sigma-$ideal.

$(2)\Rightarrow(3)$ Since $Ker T$ is the inverse image of the fuzzy $\sigma-$ideal $\{0\}$ in $F$, it follows that $Ker T$ is a fuzzy $\sigma-$ideal in $E$.

$(3)\Rightarrow(1)$Assume that $Ker T$ a fuzzy $\sigma-$ideal, and $y_{n}\uparrow y$ with $y_{n}\in E^{+}$. It suffices to show that $Ty_{n}\uparrow Ty$ with $Ty_{n}\in F^{+}$. Assume that $u(x,y)>{\frac{1}{2}}$ (otherwise replace $x$ by $x\wedge y$), take $v(Ty_{n},Tx)>{\frac{1}{2}}$, $v(Tx,Ty)>{\frac{1}{2}}$, and $z_{n}=y_{n}\vee x-x$, it follows that $u(0,z_{n})>{\frac{1}{2}}$, $z_{n}\uparrow x\vee y-x=y-x$. As $Tz_{n}=Ty_{n}\vee Tx-Tx=Tx-Tx$, we have $z_{n}\in Ker T$. On account of that $Ker T$ is a fuzzy $\sigma-$ideal, it follows that $y-x\in Ker T$, we have $Tx=Ty$. This shows that $Ty_{n}\uparrow Ty$.\ \ \ \ $\Box$
\end{proof}

\begin{definition}Let $(E,\mu)$ be a fuzzy Riesz space, for $w\in E^{+}$, the sequence $\{x_{n}\}$ in $E$ is said to converge w-uniformly in fuzzy order to $x$, if for any number $\varepsilon>0$ there exists an index $n(\varepsilon)$ such that $\mu(|x-x_{n}|,\varepsilon w)>{\frac{1}{2}}$ for all $n>n(\varepsilon)$.
\end{definition}

\begin{definition}Let $(E,\mu)$ be a fuzzy Riesz space, the sequence $\{x_{n}\}$ in $E$ is said to converge relatively uniformly in fuzzy order to $x\in E$ whenever $\{x_{n}\}$ converges w-uniformly in fuzzy order to $x$ for some $w\in E^{+}$.
\end{definition}

\begin{definition}The sequence $\{x_{n}\}$ is called a fuzzy w-uniform Cauchy sequence if for any $\varepsilon>0$ there exists an index $n(\varepsilon)$ such that $\mu(|x_{m}-x_{n}|,\varepsilon w)>{\frac{1}{2}}$ for all $m,n>n(\varepsilon)$.
\end{definition}

The fuzzy Riesz space $E$ is said to be uniformly fuzzy complete whenever for every $w\in E$, any fuzzy w-uniform Cauchy sequence has a fuzzy w-uniform limit.

\begin{theorem}Let $(E,\mu)$ and $(F,\nu)$ be two fuzzy Riesz spaces, $T:E\rightarrow F$ be a fuzzy Riesz homomorphism, then the following statements holds:

(1)The image of a relatively fuzzy uniform Cauchy sequence in $E$ is a relatively fuzzy uniform Cauchy sequence in $F$;

(2)If $\{Tx_{n}\}$ is a relatively fuzzy uniform Cauchy sequence in $F$, then there exists a subsequence $\{Tx_{n_{k}}\}$ and a corresponding sequence $\{y_{n}\}$ in $E$ such that $Ty_{k}=Tx_{n_{k}}$ for all $k$, and $\{y_{n}\}$ is a relatively fuzzy uniform Cauchy sequence in $F$.

(3)If $F$ is uniformly fuzzy complete, then so is $M$. In other words, any fuzzy Riesz homomorphic image of a uniformly fuzzy complete space is uniformly fuzzy complete.
\end{theorem}
\begin{proof}(1)Suppose that $\{x_{n}\}$ is a relatively fuzzy uniform Cauchy sequence, the definition implies that there exist some $w\in E^{+}$ such that $\mu(|x_{m}-x_{n}|,\varepsilon w)>{\frac{1}{2}}$ for all $m,n>n(\varepsilon)$. Thus we have $\nu(|Tx_{m}-Tx_{n}|,\varepsilon Tu)>{\frac{1}{2}}$ for all $m,n>n(\varepsilon)$, that is $\{Tx_{n}\}$ is a relatively fuzzy uniform Cauchy sequence.

(2)Suppose that $\{Tx_{n}\}$ is a relatively fuzzy uniform Cauchy sequence, it follows that there exist $w\in E^{+}$ such that $\{Tx_{n}\}$ is a fuzzy $Tw$-uniform Cauchy sequence. Let $\{Tx_{n_{k}}\}$ be a subsequence such that $n_{1}<n_{2}<\cdots$ and $\nu(|Tx_{n_{k}+1}-Tx_{n_{k}}|,2^{-k}Tw)>{\frac{1}{2}}$ for all $k$. Theorem 2.1 implies that there exist $\{z_{n}\}$ such that
$Tz_{k}=Tx_{n_{k}+1}-Tx_{n_{k}}$ and $\mu(|z_{k}|,2^{-k}w)>{\frac{1}{2}}$ for all $k$. Take $y_{k}=x_{n1}+z_{1}+\cdots+z_{k-1}$ satisfies $Ty_{k}=T{n_{k}}$ for all $k$, and the sequence $\{y_{n}\}$ is a fuzzy w-uniform Cauchy sequence.

(3)Suppose that $E$ is uniformly fuzzy complete and let $\{Tx_{n}\}$ is a fuzzy $Tw$-uniform Cauchy sequence. Part (2) implies that there is a subsequence $\{Tx_{n_{k}}\}$ and a corresponding sequence $\{y_{n}\}$ such that $Ty_{k}=Tx_{n_{k}}$ for all $k$ and $\{y_{n}\}$ is a fuzzy w-uniform Cauchy sequence. Since $E$ is uniformly fuzzy complete by hypothesis, $\{y_{n}\}$ converges w-uniformly in fuzzy order to $y\in E$, and so $\{Ty_{k}\}$ converges $Tw-$uniformly in fuzzy order to $Ty$. In other words, $\{Tx_{n_{k}}\}$ converges $Tw$-uniformly in fuzzy order to $Ty$ as $k\rightarrow\infty$.\ \ \ \ $\Box$
\end{proof}

\section{Fuzzy Riesz homomorphism and Fuzzy quotient spaces}

\begin{definition}Let $A$ be a linear subspace of the (real or complex) vector space $V$, $f_{1}$ and $f_{2}$ in $V$ are equivalent whenever $f_{1}-f_{2}\in A$.
\end{definition}

\begin{definition}The set of all elements in $V$ equivalent to a given $f\in V$ is called the equivalence class of $f$ and denoted by $[f]$.
\end{definition}

Remark 1:$[f_{1}]=[f_{2}]$ if and only if $f_{1}$ and $f_{2}$ are equivalent,i.e., if and only if $f_{1}-f_{2}\in A$.

Remark 1:The linear subspace $A$ itself is one of the equivalence classes; it is the equivalence class containing the null element of $V$, so $A=[0]$. That is, $f\in A$ if and only if $[f]=[0]$.

\begin{definition}The set of all equivalence classes is called the quotient space of $V$ modulo $A$, denoted by $V/A$.
\end{definition}

\begin{definition}Let $A$ be a fuzzy ideal in fuzzy Riesz space $(E,\mu)$. Let $[f]$ and $[g]$ be two elements in $E/A$, if there exist elements $f_{1}$, $g_{1}\in E$ satisfying $f_{1}-g_{1}\in A$, then $\nu([f],[g])=1$; if there exist elements $f_{1}\in[f]$ and $g_{1}\in[g]$ satisfying $\mu(f_{1},g_{1})>{\frac{1}{2}}$, then $\nu([f],[g])={\frac{2}{3}}$; otherwise, $\nu([f],[g])=0$.
\end{definition}

Remark 1:$\nu([f],[g])>{\frac{1}{2}}$ if and only if for every $f_{1}\in[f]$ there exists an element $g_{1}\in[g]$ satisfying $\mu(f_{1},g_{1})>{\frac{1}{2}}$.

Remark 2:$\nu([f],[g])>{\frac{1}{2}}$ if and only if for every $f_{1}\in[f]$ and every $g_{1}\in[g]$ there exists an element $q\in A$, such that $\mu(q,g_{1}-f_{1})>{\frac{1}{2}}$.

\begin{theorem}If $A$ is a fuzzy ideal in the fuzzy Riesz space $E$, the fuzzy quotient space $E/A$ is a fuzzy Riesz space with respect to the fuzzy order defined in Definition 3.4.
\end{theorem}

\begin{proof}First we have to show that the fuzzy order defined in Definition is a fuzzy order.

(1) The fuzzy order relation $\mu$ is obviously reflexive.

(2) For $[f],[g]\in E/A$, and $\nu([f],[g])+\mu([g],[f])>1$. Definition 3.4 implies that $\nu([f],[g])>{\frac{1}{2}}$ and $\nu([g],[f])>{\frac{1}{2}}$, Remark 2 implies that there exist $q_{1}$, $q_{2}\in A$, such that $\mu(q_{1},g-f)>{\frac{1}{2}}$, $\mu(q_{2},f-g)>{\frac{1}{2}}$, then $\mu(f-g,-q_{1})>{\frac{1}{2}}$, and $\mu(g-f,-q_{2})>{\frac{1}{2}}$. Take $q=\sup(-q_{1},-q_{2})$, we have $\mu(|f-g|,q)>{\frac{1}{2}}$ with $q\in A$. As $A$ is a fuzzy ideal, it follows that $f-g\in A$. Thus, $[f]=[g]$.

(3) For $[f],[g],[h]\in E/A$, let $[f]\neq[h]$, $[g]\neq[f]$, $[g]\neq[h]$, if $\nu([f],[g])>{\frac{1}{2}}$ and $\nu([g],[h])>{\frac{1}{2}}$, Remark 2 implies there exist $q_{1},q_{2}\in A$, satisfying $\mu(q_{1},g-f)>{\frac{1}{2}}$, $\mu(q_{2},h-g)>{\frac{1}{2}}$, so $\mu(q_{1}+q_{2},h-f)>{\frac{1}{2}}$ with $(q_{1}+q_{2})\in A$, we have $g_{1}-f_{1}\in A$. Hence, we have $\nu([f],[h]) \geq \bigvee_{[g] \in E/A}\nu(([f], [g]) \wedge \nu([g], [h]))$.

Hence, $E/A$ with fuzzy order $u$ becomes a foset.

The vector space structure and the order structure are compatible. Indeed:

(1)Let $\nu([f],[g])>{\frac{1}{2}}$, it is easy to see that $\nu(\alpha[f],\alpha[g])>{\frac{1}{2}}$ for all $0\leq\alpha\in R$.

(2)Let $\nu([f],[g])>{\frac{1}{2}}$, choose the elements $f\in[f]$, $g\in[g]$ with $\mu(f,g)>{\frac{1}{2}}$. Fix some $h\in[h]$, note that $f,g,h\in E$ and $E$ is fuzzy Riesz space imply $\mu(f+h,g+h)>{\frac{1}{2}}$, hence $\nu([f+h],[g+h])>{\frac{1}{2}}$ holds, and so $\nu(([f]+[h]),([g]+[h]))>{\frac{1}{2}}$.

Hence, $E/A$ is a fuzzy ordered linear space with respect to the fuzzy order $\nu$.

It remains to show that $E/A$ is a fuzzy Riesz space with respect to the fuzzy order $\nu$. It suffices to show that $[f]\vee[g]$ exists for all $[f]$, $[g]$ and is equal to $[f\vee g]$.
Since $\nu([f],[f\vee g])>{\frac{1}{2}}$, $\nu([g],[f\vee g])>{\frac{1}{2}}$, we have $\nu([f]\vee[g],[f\vee g])>{\frac{1}{2}}$. Next, we need to show that any bound $[h]$ of $[f]$ and $[g]$ satisfies $\nu([f\vee g],[h])>{\frac{1}{2}}$. Take $f\in [f]$, $g\in[g]$, $h\in[h]$, then there exist $q_{1}$, $q_{2}\in A$ such that $\mu(f+q_{1},h)>{\frac{1}{2}}$ and $\mu(g+q_{2},h)>{\frac{1}{2}}$. Take $q=q_{1}\wedge q_{2}\in A$, we have $\mu(f+q,h)>{\frac{1}{2}}$ and $u(g+q,h)>{\frac{1}{2}}$, it follows that $\mu((f+q)\vee(g+q),h)>{\frac{1}{2}}$.
The proposition 4.10 of [4] implies $u((f\vee g)+q,h)>{\frac{1}{2}}$. As $q\in A$, we have $\nu([f\vee g],[h])>{\frac{1}{2}}$. This proves that $[f]\vee [g]=[f\vee g]$. The Proposition 4.3 of [4] shows that $f+g=(f\vee g)+(f\wedge g)$, hence $[f]\wedge[g]$ exists and $[f]\wedge[g]=[f\wedge g]$. The fuzzy Riesz space $E/A$ is called the quotient fuzzy Riesz space of $E$ with respect to the fuzzy ideal $A$.\ \ \ \ $\Box$
\end{proof}

\begin{theorem}If $A$ is a fuzzy ideal of a fuzzy Riesz space $E$, $T$ is the canonical projection from $E$ onto $E/A$, then $T$ is a fuzzy Riesz homomorphism.
\end{theorem}
\begin{proof}Let $x\in E$, as $\mu(x,x^{+})>{\frac{1}{2}}$ and $\mu(0,x^{+})>{\frac{1}{2}}$, it follows that $v((Tx)^{+},T(x^{+}))>{\frac{1}{2}}$ and $v(0,T(x^{+})>{\frac{1}{2}}$. On the other hand, assume that $\mu(Tx,Ty)>{\frac{1}{2}}$ and $\mu(0,Ty)>{\frac{1}{2}}$ hold in $E/A$, choose $x_{1},y_{1},y_{2}\in E$ with $\mu(x_{1},y_{1})>{\frac{1}{2}}$ and $\mu(0,y_{2})>{\frac{1}{2}}$. Since $\mu((x_{1}+(x-x_{1})),y_{1}\vee y_{2}+(x-x_{1})^{+})>{\frac{1}{2}}$, proposition 4.10 of [4] implies $\mu(x,y_{1}+(y_{2}-y_{1})^{+}+(x-x_{1})^{+})>{\frac{1}{2}}$ with $y_{1}+((y_{2}-y_{1})^{+}+(x-x_{1})^{+}\in y_{1}+A$. Moreover, from $\mu(0,y_{1}\vee y_{2}+(x-x_{1})^{+})>{\frac{1}{2}}$, it follows that $\mu(x^{+},y_{1}+((y_{2}-y_{1})^{+}+(x-x_{1})^{+})>{\frac{1}{2}}$. Hence, $\nu((Tx^{+}),Ty_{1})>{\frac{1}{2}}$. Therefore, $(Tx)^{+}=T(x^{+})$.\ \ \ \ $\Box$
\end{proof}

\begin{theorem}If $(E,\mu)$ and $(F,\nu)$ are fuzzy Riesz spaces and $T:E\rightarrow F$ is a fuzzy Riesz homomorphism, then the range of $T$ is a Riesz subspace of $F$ and the kernel $Ker(T)$ of $T$ is a fuzzy ideal in $F$.
\end{theorem}
\begin{proof}It is easy to see that the range of $T$ is a Riesz subspace of $F$. Thus, it suffices to show that kernel $Ker(T)$ of $T$ is a fuzzy ideal in $H$. To this end, let $f\in Ker(T)$, we have $Tf=0$, so $|Tf|=0$, Theorem 3.16 of [9] implies $T(|f|)=0$, hence $|f|\in Ker(T)$.

Conversely, let $|f|\in Ker(T)$, then $T(|f|)=0$, we have $|T(f)|=0$, hence $Tf=0$, showing that $f\in Ker(T)$.

Finally, take $|f|\in Ker(T)$, $\mu(|g|,|f|)>{\frac{1}{2}}$, then $T$ is a positive operator since $T$ is a fuzzy Riesz homomorphism, this shows that $\nu(T(|g|),T(|f|))>{\frac{1}{2}}$, showing that $|g|\in Ker(T)$ and $g\in Ker(T)$. This proves that $Ker(T)$ is a fuzzy ideal in $F$.\ \ \ \ $\Box$
\end{proof}

The next example shows that $E/A$ need not be fuzzy Archimedean.

Example: Let $X=l_{\infty}$ be the space of all bounded sequences. Define $\mu:X\times X\rightarrow [0,1]$ by

\begin{equation}
u(x,y)=\begin{cases}
1,&\text{if} x\equiv y;\\
\frac23,&\text{if} x_{i} \leq y_{i} \text{for all} i=1,2,\cdots,n,\cdots \text{and} x\neq y;\\
0,&\text{otherwise}.
\end{cases}
\end{equation}

where $x=(x_{1},x_{2},\cdots,x_{n},\cdots)$, and $y=(y_{1},y_{2},\cdots,y_{n},\cdots)$.

Hence, $X$ is a fuzzy ordered linear space.

Consider $A$ the principal ideal generated by the element $x=(1,1/2^{2},\cdots,1/n^{2},\cdots)$. Let $e=(1,1,\cdots)$, and $y=(1,1/2,\cdots,1/n,\cdots)$, fix some $k$, and note that $\mu(1/n,1/ke(n))>{\frac{1}{2}}$ for all $n\geq k$. This implies that $\nu([y],1/k[e])>{\frac{1}{2}}$ with $[y]\neq[0]$ holds in $E/A$ for all $k=1,2,\cdots$. Corollary 4.3 of [7] implies that $E/A$ is not fuzzy Archimedean.

\begin{theorem}If $A$ is a fuzzy ideal in the fuzzy Riesz space $(E,\mu)$, $T$ is a fuzzy Riesz homomorphism from $E$ onto $E/A$, then the following statements are equivalent:

(1)$E/A$ is a fuzzy Archimedean space;

(2)$A$ is uniformly fuzzy closed;

(3)If $0\leq x_{n}\in A$ for $n=1,2,\cdots$ and the sequence $\{x_{n}\}$ is increasing and converges relatively uniformly in fuzzy order
to $x$, then $x\in A$;

(4)If $x,w\in E^{+}$ and $(nu-w)^{+}\in A$ for $n=1,2,\cdots$, then $x\in A$.
\end{theorem}
\begin{proof}$(1)\Rightarrow(2)$Since $E/A$ is fuzzy Archimedean, let $w\in E^{+}$, $\{x_{n}\}\in A$ converge  w-uniformly in fuzzy order to $x$. Given $\varepsilon>0$, since $\mu(|x-x_{n}|+|x_{n}|,\varepsilon w+|x_{n}|)>{\frac{1}{2}}$, it follows that $\mu(|x|,\varepsilon w+|x_{n}|)>{\frac{1}{2}}$, showing that $\nu(T|x|,\varepsilon Tw)>{\frac{1}{2}}$.
Since $E/A$ is fuzzy Archimedean, it follows that $T|x|=0$. This shows that $x\in A$, and hence $A$ is uniformly fuzzy closed.

$(2)\Rightarrow(3)$Evident.

$(3)\Rightarrow(4)$Let $x,w\in E^{+}$ and $(nx-w)^{+}\in A$ for $n=1,2,\cdots$. Since $\mu(0,x-(x-n^{-1}w)^{+})>{\frac{1}{2}}$ and $\mu((x-(x-n^{-1}w)^{+}),|x-(x-n^{-1}w)|)>{\frac{1}{2}}$, it follows that $\mu((x-(x-n^{-1}w)^{+}),n^{-1}w)>{\frac{1}{2}}$. Take $w_{n}=(x-n^{-1}w)^{+}$, then the increasing sequence converges w-uniformly
in fuzzy order to $x$. In view of $w_{n}\in A$, we have $x\in A$.

$(4)\Rightarrow(1)$Let $x,y\in E^{+}$ and $\nu(nTx,Ty)>{\frac{1}{2}}$, then $0=(nTx-Ty)^{+}=T\{(nx-y)^{+}\}$, it follows that $(nx-y)^{+}\in A$. The hypothesis implies that $x\in A$, that is $Tx=0$. This shows that $E/A$ is fuzzy Archimedean.\ \ \ \ $\Box$
\end{proof}

\begin{theorem}Let $(E,u)$ be a fuzzy Riesz space, $T:E\rightarrow E/A$ be a fuzzy Riesz homomorphism, then the following statements holds:

(1)If $A$ is a fuzzy ideal, then $E/A$ is fuzzy Archimedean if and only if $A$ is uniformly fuzzy closed.

(2)If $x,y\in E^{+}$, there exists a natural number $m$(depending upon $x$ and $y$) such that $x\wedge ny=x\wedge my$ for $n=m+1,\cdots$
\end{theorem}

\begin{proof}$(1)\Rightarrow(2)$Let $E/A$ be fuzzy Archimedean, and $x,y\in E^{+}$ such that $\{\inf(x,ny):n=1,2,\cdots\}$ contains infinitely many distinct members.
Let $A$ be the fuzzy ideal generated by the elements $w_{n}=nx-\inf(x,ny);n=1,2,\cdots$ The sequence $\{w_{n}:n=1,2,\cdots\}$ is increasing. Theorem 4.2 of [8] implies that $f\in A$ if and only if there exist natural numbers $n$ and $k$ satisfying $\mu(|f|, kw_{n})>{\frac{1}{2}}$. We will show that $y$ is no member of $A$. Indeed, if there exist $n$ and $k$ satisfying $\mu(y,kw_{n})>{\frac{1}{2}}$, it follows that $\mu((kx\wedge kny),(kn-1)y)>{\frac{1}{2}}$, $\mu((x\wedge kny),(kx\wedge kny))>{\frac{1}{2}}$, implying $\mu((x\wedge (kn+1)y),((x+y)\wedge(kn+1)y))>{\frac{1}{2}}$, $\mu(((x+y)\wedge(kn+1)y),kny)>{\frac{1}{2}}$. For $j\geq kn$, we have $\mu(\{x\wedge(j+1)y\},jy)>{\frac{1}{2}}$. Thus $(x\wedge jy)=(x\wedge(j+1)y)$, contradicting the assumption that the set $\{(x\wedge ny):n=1,2,\cdots\}$ contains infinitely many distinct elements. Hence, $y$ is no member of the ideal $A$. Since $T:E\rightarrow E/A$, and $\mu(ny,x+\{ny-(x\wedge ny)\})>{\frac{1}{2}}$, it follows that $\nu(nTy,Tx)>{\frac{1}{2}}$ for $n=1,2,\cdots$. As $y$ is no member of $A$, we have $Ty\neq[0]$. This contradicts the hypothesis that $E/A$ is fuzzy Archimedean. Hence, the set $\{(x\wedge ny):n=1,2,\cdots\}$ contains only a finite number of elements.

$(2)\Rightarrow(1)$Let $A$ be a fuzzy ideal in $E$, and $x,y\in E^{+}$ satisfying $(nx-y)^{+}\in A$ for $n=1,2,\cdots$. Theorem 3.8 implies that it need to show that $x\in A$. Note that $(nx-y)^{+}=(nx-y)\wedge0$, Theorem 4.8 of [5] implies that $(nx-y)\wedge0=(nx\wedge y)-y$, Proposition 4.3 of [4] implies that $nx\wedge y-y=nx-nx\wedge y$. By hypothesis, there exists a natural number $m$ such that $nx\wedge y=mx\wedge y$ holds for all $n\geq m$. We have $\{(n+1)x-y\}^{+}-(nx-y)^{+}=(n+1)x-(n+1)x\wedge y-(nx-nx\wedge y)=x$, this implies that $x\in A$ since ${(n+1)x-y}^{+}\in A,$ and $(nx-y)^{+}\in A$.\ \ \ \ $\Box$
\end{proof}

\begin{theorem}Let $(E,\mu)$ be a fuzzy Riesz space, then the following statements holds:

(2)If $x,y\in E^{+}$, there exists a natural number $m$(depending upon $x$ and $y$) such that $x\wedge ny=x\wedge my$ for $n=m+1,\cdots$

(3)Every fuzzy principal ideal in $E$ is a fuzzy projection band.
\end{theorem}
\begin{proof}$(2)\Rightarrow(3)$Let $y\in E^{+}$, take $A_{y}$ be the fuzzy principal ideal generated by $y$. Corollary 5.1 of [8] implies that every positive $x$ in the fuzzy principal band generated by $y$ satisfies $(x\wedge ny\uparrow x:n=1,2,\cdots)$. The hypothesis implies that there exist a natural number $m$ such that $x\wedge ny=x\wedge my$ for all $n\geq m$. It follows that $x=x\wedge my$, thus $u(x,my)>{\frac{1}{2}}$, showing that $x\in A_{y}$. Thus, the fuzzy ideal $A_{y}$ is the fuzzy band generated by $y$. By Corollary 5.1 of [8], it suffices to show $\sup((x\wedge ny):n=1,2,\cdots)$ exists for every $x\in E^{+}$. Take $x\in E^{+}$, the hypothesis implies that there exist a natural number $m$ such that $x\wedge ny=x\wedge my$ holds for all $n\geq m$, and hence $\sup(x\wedge ny)=x\wedge my$.

$(3)\Rightarrow(2)$Let $y\in E^{+}$, by hypothesis, the fuzzy ideal $A_{y}$ generated by $y$ is the fuzzy band generated by $y$, and the fuzzy band $A_{y}$ is a fuzzy projection band. Take $x\in E^{+}$, and $x_{1}\in A_{y}$, $x_{2}\in (A_{y})^{d}$, we have $x_{2}\wedge ny=0$, showing that $x\wedge ny=x_{1}\wedge ny$. As $x_{1}\in A_{y}$, it follows that $\mu(x_{1}, my)>{\frac{1}{2}}$ for some natural number $m$, so $x_{1}\wedge ny=x_{1}=x_{1}\wedge my$ holds for all $n\geq m$. Thus, $x\wedge ny=x_{1}\wedge ny=x_{1}\wedge my=x\wedge my$ holds for all $n\geq m$.\ \ \ \ $\Box$
\end{proof}

\section{The fuzzy order continuous properties of fuzzy Riesz homomorphism}

\begin{definition}In a fuzzy Riesz space $(E,\mu)$, a net $\{x_{\alpha}\}$ is fuzzy order convergent to $x$, whenever there exists a net $\{y_{\alpha}\}$ with the same indexed set satisfying $\mu(|x_{\alpha}-x|,y_{\alpha})>{\frac{1}{2}}$ with $y_{\alpha}\downarrow0$, denoted by $x_{\alpha}\overset{O_{F}}{\longrightarrow} x$.

\end{definition}

\begin{definition}An operator $T:(E,\mu)\rightarrow (F,\nu)$ between two fuzzy Riesz spaces is said to be fuzzy $\sigma-$order continuous, whenever $x_{\alpha}\overset{O_{F}}{\longrightarrow} 0$ in $E$ implies $Tx_{\alpha}\overset{O_{F}}{\longrightarrow} 0$ in $F$.
\end{definition}

\begin{definition}An operator $T:(E,\mu)\rightarrow (F,\nu)$ between two fuzzy Riesz spaces is said to be fuzzy order continuous, whenever $x_{n}\overset{O_{F}}{\longrightarrow} 0$ in $E$ implies $Tx_{n}\overset{O_{F}}{\longrightarrow} 0$ in $F$.
\end{definition}

\begin{definition}A fuzzy Riesz space is said to have the fuzzy $\sigma-$order continuity property whenever every fuzzy positive operator from $E$ into an arbitrary fuzzy Archimedean Riesz space is fuzzy $\sigma-$order continuous.

A fuzzy Riesz space is said to have fuzzy order continuity property whenever every fuzzy positive operator from $E$ into an arbitrary fuzzy Archimedean Riesz space is fuzzy order continuous.
\end{definition}

\begin{definition}Let $(E,\mu)$ be a fuzzy Riesz space. A fuzzy Riesz subspace $G$ of $E$ is said to be fuzzy order dense in $E$ if for every nonzero positive element $x\in E$, there exists a nonzero element $g\in G$ such that $\mu(g,x)>{\frac{1}{2}}$.
\end{definition}

\begin{definition}A fuzzy Dedekind complete Riesz space $F$ is said to be a fuzzy Dedekind completion of the Riesz space $E$ whenever $E$ is fuzzy Riesz isomorphic to a fuzzy order dense majoring Riesz subspace of $F$.
\end{definition}

\begin{theorem}Let $(E,\mu)$ and $(F,\nu)$ be two fuzzy Riesz spaces with $F$ fuzzy Dedekind complete, $T:E\rightarrow F$ be a fuzzy order bounded operator, then the following statements hold:

(1)$T$ is fuzzy order continuous if and only if the fuzzy null ideal $N_{S}$ is a fuzzy band for every fuzzy operator $S$ in the fuzzy ideal $A_{T}$ generated by $T$ in $FL_{b}(E,F)$.

(2)$T$ is fuzzy $\sigma-$order continuous if and only if the fuzzy null ideal $N_{S}$ is a fuzzy $\sigma-$ideal for $S\in A_{T}$.
\end{theorem}
\begin{proof}(1)Assume $T$ is fuzzy order continuous, Theorem 3.9 of [11] implies that the fuzzy null ideal $N_{S}$ is a fuzzy band for $S\in A_{T}$. Next, assume that $T\geq0$, take $0\leq x_{\alpha}\uparrow x$ in $E$, $0\leq Tx_{\alpha}\uparrow y$ in $F$ with $\nu(y,Tx)>{\frac{1}{2}}$. To show that $T$ is fuzzy order continuous, it suffices to show that $y=Tx$. Let $0<\varepsilon<1$, for each $\alpha$, and $x\in E^{+}$, $T_{\alpha}=T_{A}$ on $A$, and $T_{\alpha}=0$  on $A^{d}$, where $T_{A}(x)=\sup\{T(z):0\leq z\leq x, z\in A\}$, and $A$ is the fuzzy ideal generated by $(\varepsilon x-x_{\alpha})^{+}$. Thus, $0\leq T_{\alpha}\downarrow\leq T$, and $T_{\alpha}(\varepsilon x-x_{\alpha})^{-}=0$ for each $\alpha$. Let $T_{\alpha}\downarrow S\geq0$ in $L_{b}(E,F)$, as $S\in A_{T}$, it follows that $S(\varepsilon x-x_{\alpha})^{-}=0$, this implies that $\{(\varepsilon x-x_{\alpha})^{-}\}\subseteq N_{S}$. On the other hand, in view of $(\varepsilon x-x_{\alpha})^{-}\uparrow (1-\varepsilon)x$, together with the fact that $N_{S}$ is a fuzzy band, we have $x\in N_{S}$. Therefore, $Sx=0$. Since $T(\varepsilon x-x_{\alpha})^{+}=T_{\alpha}(\varepsilon x-x_{\alpha})^{+}$ㄛand $\nu(T_{\alpha}(\varepsilon x-x_{\alpha})^{+},T_{\alpha}(x))>{\frac{1}{2}}$, and consider that $\mu(0,x-x_{\alpha})>{\frac{1}{2}}$, $\mu(x-x_{\alpha},(1-\varepsilon)x+(\varepsilon x-x_{\alpha})^{+})>{\frac{1}{2}}$, we have $\nu(Tx-y,T(x-x_{\alpha}))>{\frac{1}{2}}$, $\nu(T(x-x_{\alpha}),(1-\varepsilon)Tx+T(\varepsilon x-x_{\alpha})^{+})>{\frac{1}{2}}$. This implies that $\nu(Tx-y,(1-\varepsilon)Tx+T_{\alpha}(x))>{\frac{1}{2}}$. As $T_{\alpha}(x)\downarrow S(x)=0$, we have $\nu(Tx-y,(1-\varepsilon)Tx)>{\frac{1}{2}}$ for all $0<\varepsilon<1$. Therefore, $y=Tx$.

(2)It can be proven similarly.\ \ \ \ $\Box$\end{proof}

\begin{theorem}For a fuzzy Riesz space $(E,\mu)$, the following statements are equivalent:

(1)$E$ has the fuzzy $\sigma-$order continuity property;

(2)Every fuzzy lattice homomorphism $T$ from $E$ into a fuzzy Archimedean Riesz space is fuzzy $\sigma-$order continuous;

(3)Every uniformly fuzzy closed ideal of $E$ is a fuzzy $\sigma-$ideal.
\end{theorem}
\begin{proof}$(1)\Rightarrow(2)$Since $T(x^{+})=T(x\vee0)=T(x)\vee T(0)=[T(x)]^{+}$, it follows that $v(0,Tx)>{\frac{1}{2}}$, we have every fuzzy lattice homomorphism is a fuzzy positive operator.

(2)$\Rightarrow(3)$Let $A$ be a uniformly fuzzy closed ideal of $E$, Theorem 3.8 implies that $E/A$ is a fuzzy Archimedean Riesz space. Since the canonical projection of $E$ onto $E/A$ is a fuzzy lattice homomorphism, it follows that it is fuzzy $\sigma-$order continuous. Theorem 2.10 implies that the kernel $A$ is a fuzzy $\sigma-$ideal.

(3)$\Rightarrow(1)$Let $F$ be a fuzzy Archimedean Riesz space, $T:E\rightarrow F$ be a fuzzy positive operator. Since $F$ is a fuzzy order dense Riesz subspace of its fuzzy Dedekind completion $F^{\delta}$, it follows that $T:E\rightarrow F$ is fuzzy $\sigma-$order continuous if and only if $T:E\rightarrow F^{\delta}$ is fuzzy $\sigma-$order continuous. Hence we assume that $F$ is fuzzy Dedekind complete.

Take $S:E\rightarrow F$ be a fuzzy order bounded operator satisfying $|S|\leq T$. Let $\{x_{n}\}\subseteq N_{T}$ satisfies $\mu(|x-x_{n}|,\varepsilon_{n}u)>{\frac{1}{2}}$ with $\varepsilon_{n}\downarrow0$ and $u\in E^{+}$, since $|T|(|x|)=|T|(|x|-|x_{n}|)$ and $\nu(|T|(|x|-|x_{n}|),|T|(|x-x_{n}|))>{\frac{1}{2}}$, $\nu(|T|(|x-x_{n}|),\varepsilon_{n}|T|u)>{\frac{1}{2}}$, it follows that $|T|(|x|)=0$. This shows that $N_{S}$ is uniformly fuzzy closed. The hypothesis implies that $N_{S}$ is a fuzzy $\sigma-$ideal. Thus, the null ideal of every operator in $A_{T}$ is a fuzzy $\sigma-$ideal. Thus, Theorem 4.7 implies the operator $T$ is fuzzy $\sigma-$order continuous.\ \ \ \ $\Box$\end{proof}

\begin{definition}A fuzzy Riesz space $E$ is said to have the fuzzy countable sup property whenever $\sup A$ exists in $E$, then there exists at most countable subset $B$ of $A$ with $\sup B=\sup A$. A fuzzy Riesz space has the fuzzy countable sup property if and only if $0\leq H\uparrow h$ implies the existence of at most countable subset $G$ of $H$ with $\sup G=h$.
\end{definition}

\begin{theorem}A fuzzy Archimedean Riesz space has the fuzzy countable sup property if and only if every fuzzy $\sigma-$ideal is a fuzzy band.
\end{theorem}
\begin{proof}If a fuzzy Riesz space has the fuzzy countable sup property, then every fuzzy $\sigma-$ideal is a fuzzy band. Conversely, let $0\leq H\uparrow h$ hold in $E$ with $H$ uncountable. We need to show that there exists a countable subset $G$ of $H$ with $\sup G=h$. Fix $0<\varepsilon<1$ㄛtake $A=\{y\in E:\exists$ a countable $G\subseteq H$ with $\inf\{|y|\wedge(\varepsilon h-g)^{+}:g\in G\}=0\}$. It suffices to show that $A$ is a fuzzy $\sigma-$ideal.

Clearly, $A$ is a fuzzy solid set. Since $y\in A$, $\alpha\in R$, it follows that $\alpha y\in A$. Besides, if $y,z\in A$, pick two countable subsets $P=\{p_{1},p_{2},\cdots\}$ and $Q=\{q_{1},q_{2},\cdots\}$ of $H$ with $\inf\{|y|\wedge(\varepsilon h-p)^{+}:p\in P\}=$$\inf\{|z|\wedge(\varepsilon h-q)^{+}:q\in Q\}=0$. Since $H$ is directed upward, for each $n$ and $m$, there exists some $w_{n,m}$ in $H$ with $\mu(p_{n}\vee q_{m}, w_{n,m})>{\frac{1}{2}}$. Since $\mu(|y+z|\wedge (\varepsilon h-w_{n,m})^{+},|y|\wedge(\varepsilon h-p_{n})^{+}+|z|\wedge(\varepsilon h-q_{m})^{+})>{\frac{1}{2}}$, it follows that the countable set $W=\{w_{n,m}:n,m=1,2\cdots\}$ of $H$ satisfies $\inf\{|y+z|\wedge (\varepsilon h-w)^{+}:w\in W\}=0$. Therefore, $y+z\in A$, so $A$ is a fuzzy ideal of $E$.

We claim that $A$ is a fuzzy $\sigma-$ideal. Let $\{y_{n}\}\subseteq A$ satisfying $0\leq\{y_{n}\}\uparrow y$ in $E$. For each $n$, pick a countable subset $G_{n}$ of $H$ with $\inf\{y_{n}\wedge(\varepsilon h-g)^{+}:g\in G_{n}\}=0$. Consider the countable subset $G=\cup_{n=1}^{\infty} G_{n}$ of $H$. Suppose $p\in H^{+}$, $\mu(p,y\wedge(\varepsilon h-g)^{+})>{\frac{1}{2}}$ for all $g\in G$, it follows that $\mu(y_{n}\wedge p,y_{n}\wedge(\varepsilon h-g)^{+})>{\frac{1}{2}}$ holds for all $g\in G_{n}$, together with the fact $\mu(0,y_{n}\wedge p)>{\frac{1}{2}}$, it follows that $y_{n}\wedge p=0$ for all $n$. As $y_{n}\wedge p\uparrow y\wedge p$, we have $p=0$. This shows that $\inf\{y\wedge(\varepsilon h-g)^{+}:g\in G\}=0$, and $y\in A$. Therefore, $A$ is a fuzzy $\sigma-$ideal, the hypothesis implies it is a fuzzy band.

As $(p-\varepsilon h)^{+}\in A$ for $p\in H$, and $\{(p-\varepsilon h)^{+}:p\in H\}\uparrow(1-\varepsilon)h$, it follows that $x\in A$. Therefore, there exists a countable subset $G$ of $H$ satisfying $\inf\{h\wedge(\varepsilon h-g)^{+}:g\in G\}=\inf\{(\varepsilon h-g)^{+}:g\in G\}=0$.

For $\varepsilon_{n}=1-n^{-1}(n=2,3,\cdots)$, choose a countable subset $G_{n}$ of $H$ with $\inf\{(\varepsilon_{n}h-g)^{+}:g\in G_{n}\}=0$, and consider the countable subset $G=\cup^{\infty}_{n=2}G_{n}$ of $H$. Take $\mu(y,h-g)>{\frac{1}{2}}$, $\mu(0,y)>{\frac{1}{2}}$ hold for all $g\in G$. Since $\mu(h-g,(1-\varepsilon_{n})h+(\varepsilon_{n}h-g)^{+})>{\frac{1}{2}}$ㄛit follows that $\mu(y,n^{-1}x+(\varepsilon_{n}h-g)^{+})>{\frac{1}{2}}$, that is $\mu(y,n^{-1}h+\inf\{(\varepsilon_{n}h-g)^{+}:g\in G_{n}\})>{\frac{1}{2}}$, thus $\mu(y,n^{-1}h)>{\frac{1}{2}}$. Therefore, $y=0$. This shows that $\inf\{h-g:g\in G\}=0$. We conclude that $\sup G=h$, proving that $E$ has the fuzzy countable sup property.
\ \ \ \ $\Box$\end{proof}

\begin{theorem}For a fuzzy Riesz space $E$, the following statements are equivalent:

(1)$E$ has the fuzzy order continuity property;

(2)Every fuzzy lattice homomorphism from $E$ into a fuzzy Archimedean Riesz space is fuzzy order continuous;

(3)Every uniformly fuzzy closed ideal of $E$ is a fuzzy band;

(4)$E$ has the fuzzy $\sigma-$order continuity property and the fuzzy countable sup property.
\end{theorem}
\begin{proof}$(1)\Rightarrow(2)$Clearly.

(2)$\Rightarrow(3)$Let $A$ be a uniformly fuzzy closed ideal of $E$. Theorem 3.8 implies $E/A$ is a fuzzy Archimedean Riesz space. Since the canonical projection of $E$ onto $E/A$ is a fuzzy lattice homomorphism, by the hypothesis, it is fuzzy order continuous, Theorem 4.7 implies that $A$ is a fuzzy band.

(3)$\Rightarrow(4)$Theorem 4.8 implies the fuzzy Riesz space has the fuzzy $\sigma-$order continuity property. Since $E$ is fuzzy Archimedean, by the hypothesis, we have every fuzzy $\sigma-$ideal of $E$ is a fuzzy band. The conclusion follows from Theorem 4.10.

(4)$\Rightarrow(1)$Clearly.\ \ \ \ $\Box$\end{proof}

\begin{definition}A fuzzy vector subspace $G$ of a fuzzy Riesz space $(E,\mu)$ is fuzzy majorizing $E$ whenever for each $x\in E$ there exists some $y\in G$ with $\mu(x,y)>{\frac{1}{2}}$.
\end{definition}

\begin{theorem}Let $(E,\mu)$ and $(F,\nu)$ be two fuzzy Riesz spaces with $F$ fuzzy Dedekind complete. If $M$ is a fuzzy majorizing Riesz subspace of $E$ and $T:M\rightarrow F$ is a fuzzy lattice homomorphism, then $T$ extends to all of $E$ as a fuzzy lattice homomorphism.
\end{theorem}

\begin{proof}Define $\theta(x)=\inf\{Tz:z\in M,\mu(x,z)>{\frac{1}{2}}\}$ for $x\in E$. It is clearly that $\theta$ is fuzzy sublinear and preserves finite suprema. Moreover, $\theta(x+z)=\theta(x)+\theta(z)$ for all $x\in E$ and $z\in M$. For $x_{0}\in E$, let $R=\{\sup(z_{1}+t_{1}x_{0},\cdots,z_{n}+t_{n}x_{0}):z_{i}\in M,t_{i}\geq0,n\in N\}$ㄛthen $R$ is a cone in $E$, and $R$ contains finite suprema of its elements. Next, we will show that $\theta$ is additive on $R$. Pick $\alpha\in[0,\infty]$, $r_{1},r_{2}\in R$ defined by
$r_{1}=\sup\{z_{1}+t_{1}x_{0},\cdots,z_{n}+t_{n}x_{0}\}$, and $r_{2}=\sup\{y_{1}+s_{1}x_{0},\cdots,y_{m}+s_{m}x_{0}\}$, $n,m\in N$, $z_{i},y_{j}\in M$, $t_{i},s_{j}\geq0$. It follows that $r_{1}+r_{2}=\sup\{z_{i}y_{j}+(t_{i}+s_{j})x_{0}:i=1,\cdots,n;j=1,\cdots,m\}\in R$. Moreover, $\theta(r_{1}+r_{2})=\sup\{\theta((z_{i}+y_{j})+(t_{i}+s_{j})x_{0}):1\leq i\leq n;1\leq j\leq m\}$=$\sup\{\theta(z_{i}+t_{i}x_{0})+\theta(y_{j}+s_{j}x_{0}):1\leq i\leq n;1\leq j\leq m\}$=$\theta(r_{1})+\theta(r_{2})$. For $p_{1},p_{2},q_{1},q_{2}\in R$ satisfying $p_{1}-p_{2}=q_{1}-q_{2}$, the additivity of $\theta$ on $R$ implies that $\theta(p_{1})-\theta(p_{2})=\theta(q_{1})-\theta(q_{2})$. Define $Sx=\theta(p_{1})-\theta(p_{2})$ for $x=p_{1}-p_{2}\in R-R$. Hence, in view of $\theta(x)=Tx$ for all $x\in M$, we see that $S$ extends $T$ to all of $R-R$. Since $\theta$ preserves finite suprema, it follows that $S(v_{1}\vee v_{2})=S(v_{1})\vee S(v_{2})$. Hence, let $v_{1},v_{2},w_{1},w_{2}\in R$, since $(p_{1}-p_{2})\vee(q_{1}-q_{2})+p_{2}+q_{2}=(p_{1}+q_{2})\vee(p_{2}+q_{1})$,
we obtain $S((p_{1}-p_{2})\vee(q_{1}-q_{2}))=$$S((p_{1}+p_{2})\vee(p_{2}+q_{1})-(p_{2}+q_{2}))$=$S(p_{1}+q_{2})\vee S(p_{2}+q_{1})-Sp_{2}-Sq_{2}$=$S(p_{1}-p_{2})\vee S(q_{1}-q_{2})$. Thus, $S$ is a fuzzy lattice homomorphism on $R-R$. This completes the proof.\ \ \ \ $\Box$\end{proof}

\begin{definition}A fuzzy Riesz subspace $G$ of a fuzzy Riesz space $(E,\mu)$ is fuzzy order dense in $E$ whenever for each $x\in E^{+}$ there exists some $y\in G^{+}$ with $\mu(y,x)>{\frac{1}{2}}$.
\end{definition}

\begin{definition}A fuzzy Riesz space is called laterally complete whenever every set of pairwise disjoint positive elements has a supremum.
\end{definition}

\begin{theorem}Let $T:H\rightarrow F$ be a fuzzy order continuous lattice homomorphism from a fuzzy Dedekind complete Riesz space $H$ into a fuzzy Archimedean laterally complete Riesz space $(F,\nu)$. If $H$ is a fuzzy order dense Riesz subspace of a fuzzy Archimedean Riesz space $(E,\mu)$, then the formula $T(x)=\sup\{T(h):h\in H, 0\leq h\leq x\}$, $x\in E^{+}$, defines a fuzzy extension of $T$ from $E$ into $F$, which is a fuzzy order continuous lattice homomorphism.
\end{theorem}
\begin{proof}Pick $x\in E^{+}$, consider the nonempty set $D=\{h\in H^{+}:\mu(y,x)>{\frac{1}{2}}\}$. It suffices to show the existence of $\sup T(D)$. First, we need to show for each $r\in H^{+}$, there exists some $k\in H^{+}$ with $\mu(k,r)>{\frac{1}{2}}$ and a positive integer $n$ satisfying $\mu(P_{k}(h),nk)>{\frac{1}{2}}$ for all $h\in D$, where $P_{k}$ is the fuzzy order projection of $H$ onto the fuzzy band generated by $k$ in $H$. Since $E$ is fuzzy Archimedean, there exists some integer $n$ with $\mu(0,(nr-x)^{+})>{\frac{1}{2}}$. Pick $g\in H^{+}$ with $\mu(g,(nr-x)^{+})>{\frac{1}{2}}$, and note that $g\perp(h-nr)^{+}$ holds for all $h\in D$. From $\mu(g,nr)>{\frac{1}{2}}$, we see that $\mu(0,g\wedge r)>{\frac{1}{2}}$, and so $k=P_{g}(r)\in H$ satisfies $\mu(k, r)>{\frac{1}{2}}$. Now for each $h\in D$ the relation $[P_{g}(h)-nP_{g}(r)]^{+}=[P_{g}(h)-nk]^{+}=P_{g}(h-nr)^{+}\in B_{g}\cap B_{g}^{d}={0}$ implies that $[P_{g}(h)-nk]^{+}=0$, and hence $\mu(P_{g}(h), nk)>{\frac{1}{2}}$ holds for all $h\in D$. Since $P_{k}(h)=\sup\{h\wedge ik:i=1,2,\cdots\}\in B_{g}$, it follows that $P_{k}(h)=P_{g}(P_{k}(h))=P_{k}(P_{g}(h))$, and $\mu((P_{k}(P_{g}(h)), nP_{k}(k)=nk)>{\frac{1}{2}}$, this implies that $\mu(P_{k}(h),nk)>{\frac{1}{2}}$ holds for all $h\in D$.

Next, Zorn'n lemma, together with the fact that $H$ is fuzzy order denseness in $E$ imply that there exists a maximal disjoint system of strictly positive elements $\{k_{i}:i\in I\}$ and a positive integer $n_{i}$, satisfying $\mu(P_{k_{i}}(h), n_{i}k_{i})>{\frac{1}{2}}$ for all $h\in D$. If $r_{i}=n_{i}k_{i}$, then $\{r_{i}:i\in I\}$ is the maximal disjoint system of strictly positive elements of $H$ such that $\mu(P_{r_{i}}(y), r_{i})>{\frac{1}{2}}$ holds for all $i\in I$ and all $h\in D$.

Next, we need to show there exists a pairwise disjoint set $\{h_{i}:i\in I\}\subseteq D$ satisfying $x=\sup(h_{i})$ in $E$. Since $H$ is fuzzy Dedekind complete and $\mu(P_{r_{i}}(h), r_{i})>{\frac{1}{2}}$ holds for all $h\in D$, it follows that $h_{i}=\sup\{P_{r_{i}}(h):h\in D\}$ exists in $H$ satisfying $\mu(h_{i}, r_{i})>{\frac{1}{2}}$. Now let $h\in D$, $k\in H^{+}$, we have $\mu(P_{r_{i}}(h), h-k)>{\frac{1}{2}}$ for all $i$, then $\mu(k, h-P_{r_{i}}(y))>{\frac{1}{2}}$ with $h-P_{r_{i}}(h)\in B_{r_{i}}^{d}$ hold for each $i$, so $k\wedge r_{i}=0$ for all $i$. Since $\{r_{i}:i\in I\}$ is a maximal disjoint system, it follows that $k=0$, therefore $h=\sup\{P_{r_{i}}(h):i\in I\}$ holds in $G$ for each $h\in D$.

Note that for each $h\in D$, we have $\mu(P_{r_{i}}(h), h)>{\frac{1}{2}}$, $\mu(h, x)>{\frac{1}{2}}$, hence $\mu(h_{i}, x)>{\frac{1}{2}}$. On the other hand, if $\mu(h_{i}, k)>{\frac{1}{2}}$, then $\mu(P_{r_{i}}(h), k)>{\frac{1}{2}}$ holds for all $h\in D$, it follows that $\mu(\sup\{P_{r_{i}}(h):i\in I\}, k)>{\frac{1}{2}}$ for all $h\in D$. Therefore, $\mu(x, k)>{\frac{1}{2}}$, hence $x=\sup\{h_{i}\}$ holds in $E$.

Finally, we will show that if $\{h_{i}:i\in I\}$ is a pairwise disjoint subset of $D$ with $x=\sup\{h_{i}\}$, then $\sup T(D)=\sup\{T(h_{i})\}$ holds in $F$.

Since $T:H\rightarrow F$ is a fuzzy lattice homomorphism, $\{T(h_{i}):i\in I\}$ is a pairwise disjoint subset of $F^{+}$. Since $F$ is a fuzzy Archimedean laterally complete Riesz space, the element $s=\sup\{T(h_{i}):i\in I\}$ exists in $F$. Let $h\in D$, since $T$ is a fuzzy order continuous lattice homomorphism, it follows that $h=h\wedge x=h\wedge\sup\{h_{i}:i\in I\}=\sup\{h\wedge h_{i}:i\in I\}$, we have $T(h)=\sup\{T(h\wedge h_{i}):i\in I\}=\sup\{T(h)\wedge T(h_{i}):i\in I\}=T(h)\wedge \sup\{T(h_{i}):i\in I\}=T(h)\wedge s$, this implies $\nu(T(h),s)>{\frac{1}{2}}$. Thus, $s=\sup T(D)$.\ \ \ \ $\Box$\end{proof}

\section{Factorization by fuzzy Riesz homomorphisms}

\begin{theorem}Let $(E,\mu)$, $(F,\nu)$, $(H,\omega)$ be fuzzy Riesz spaces with $F$ fuzzy Dedekind complete, $Q:E\rightarrow H$ be a fuzzy lattice homomorphism, $S:H\rightarrow F$ be a fuzzy positive linear operator, then every fuzzy positive linear $T:E\rightarrow F$ satisfies $T\leq S\circ Q$ admits a factorization $T=S_{1}\circ Q$ where $S_{1}:H\rightarrow F$ is linear and $0\leq S_{1}\leq S$.
\end{theorem}

\begin{proof}Let $H_{1}=QE$, then $H_{1}$ is a fuzzy sublattice of $H$. Define $S_{0}:H_{1}\rightarrow F$ by $S_{0}(Qx)=Tx$ for $x\in E$. Next, suppose $Qx=0$, as $\nu(|Tx|,T|x|)>{\frac{1}{2}}$, $\nu(T|x|,SQ(|x|))>{\frac{1}{2}}$, and $SQ(|x|)=S(|Qx|)=0$, we have $Tx=0$. Thus, $S_{0}$ is well defined, and is obviously a linear mapping.

Let $p:H\rightarrow F$ be defined by $p(y)=Sy^{+}$. Suppose $y_{1}$, $y_{2}\in H$, as $\omega((y_{1}+y_{2})^{+},(y_{1}^{+}+y_{2}^{+}))>{\frac{1}{2}}$, it follows that $\nu(S(y_{1}+y_{2})^{+},(Sy_{1}^{+}+Sy_{2}^{+}))>{\frac{1}{2}}$, we have $\nu(p(y_{1}+y_{2}),p(y_{1})+p(y_{2}))>{\frac{1}{2}}$. Moreover, let $y\in H$, $\lambda\in R$, $\lambda\geq0$, we have $p(\lambda y)=S(\lambda y)^{+}=\lambda Sy^{+}=\lambda p(y)$. Therefore, $p$ is fuzzy sublinear. By Theorem 2.5 of [12], $S_{0}$ has a linear extension $S_{1}:H\rightarrow F$ satisfies $\nu(S_{1},p(y))>{\frac{1}{2}}$ for $y\in H$. The definition of $S_{0}$ implies that $S_{1}(Qx)=S_{0}(Qx)=Tx$ for $x\in E$. Moreover, let $y\in H^{+}$, in view of $-S_{1}y=S_{1}(-y)$, $p(-y)=S(-y)^{+}$, $\nu(S_{1}(-y),p(-y))>{\frac{1}{2}}$, we have $\nu(0,S_{1}(y))>{\frac{1}{2}}$, showing $S_{1}$ is fuzzy positive. Finally, let $y\in H^{+}$, since $\nu(S_{1}(y),p(y)=S(y))>{\frac{1}{2}}$, it follows that $S_{1}\leq S$. Consequently, $S_{1}$ has the desired properties.\ \ \ \ $\Box$\end{proof}

 REFERENCES

[1]Zadeh, L.A. ,Fuzzy Sets. Information and Control ,1965, 8, 338每353.

[2]Zadeh, L.A.ㄛ Similarity relations and fuzzy orderingㄛInformation Sciences, 1971, 177每200.

[3]Venugopalan, P. Fuzzy Ordered Sets. Fuzzy Sets and systems, 1992, 46, 221每226.

[4]Beg, I.,Islam, M., Fuzzy Riesz Spaces, J. Fuzzy Math.,1994, 2,211每241.

[5]Beg, I., Islam, M., Fuzzy ordered linear spaces. J. Fuzzy Math. 1995, 3, 659每670.

[6]Beg, I.,  $\sigma$-complete fuzzy Riesz spaces. Results in Mathematics,1997, 31, 292-299.

[7]Beg, I., Islam, M. , Fuzzy Archimedean spaces, J. Fuzzy Math.,1997, 5, 413每423.

[8]Hong, L., Fuzzy Riesz subspaces, fuzzy ideals, fuzzy bands and fuzzy band projections,
Ann. West Univ. Timis.Math. Comput. Sci. , 2015, 53, 77每108.

[9]M. Iqbal, Z. Bashir., The existence of fuzzy Dedekind completion of Archimedean fuzzy Riesz space. Computational and Applied Mathematics, 2020, 39, 116.

[10]Beg, I.,Extension of fuzzy positive linear operators,J. Fuzzy Math.,1998, 4, 849-855.

[11]M. Iqbal, Z. Bashir., A study on fuzzy order bounded linear operators in fuzzy Riesz spaces.

[12]Na Cheng, Xiao Liu, Juan Dai, Extension of fuzzy linear operators on fuzzy Riesz spaces, submit.

[13] Sesma-Sara, M ; Mesiar, R  ; Bustince, H ㄛWeak and directional monotonicity of functions on Riesz spaces to fuse uncertain data,Fuzzy Sets and Systems, 386,145-160, 2020.

[14] Mosadegh, SMSM  ; Movahednia, E , Stability of preserving lattice cubic functional equation in Menger probabilistic normed Riesz spaces,
 20(1), 2018.

[15]Sheng-Gang LiㄛOn two weaker forms of continuous order-homomorphisms,Fuzzy Sets and Systems ,101 , 469 -475,1999.

[16]S.L. Chen, Several order-homomorphisms on L-fuzzy topological spaces, J. Shaanxi Normal Univ. 16 (3),15-19, 1988.

[17]S.L. Chen, J.S. Cheng, The characterizations of semi-continuous and irresolute order-homomorphisms of fuzzes,Fuzzy Sets and Systems, 64 , 105- 112,1994.

[18]G.J. Wang, Order-homomorphisms on Fuzzes, Fuzzy Sets and Systems, 12 ,281-288,1984.

[19] G.J. Wang, Theory of topological molecular lattices,Fuzzy sets and Systems ,47 ,351-376,1992.

[20] Wojciech Bielas, Aleksander B?aszczyk,Topological representation of lattice homomorphisms,Topology and its Applications 196 , 362每378, 2015.

[21]J.W.Grzymala-Busse, W.A.Sedelow, on rough sets and information system homomorphism, Bull. pol.Acad.Technol.Sci.,36,233-239, 1988.

[22] D.Y. Li, Y.C. Ma, Invariant characters of information systems under some homomorphisms, Information Sciences ,129 , 211-220, 2000.

[23]C. Wang, C. Wu, D. Chen, W. Du, Some properties of relation information systems under homomorphisms, Applied Mathematics Letters, 21, 940-945, 2008.
[24] C. Wang, C. Wu, D. Chen, Q. Hu, C. Wu, Communication between information systems, Information Sciences ,178 , 3228-3239, 2008.

[25] Hongxiang Tang, Zhaowen Li,Invariant characterizations of fuzzy information systems under some homomorphisms based on data compression and related results,Fuzzy Sets and Systems,376 ,37-72, 2019.

[26]C.Z.Wang, D.G.Chen,Q.H.Hu, fuzzy informaion system and their homomorphisms, Fuzzy Sets and Systems, 249, 128-138, 2014.

[27]T.Bag, S.K.Samanta, Finite dimensional fuzzy normed linear spaces, J.Fuzzy Math.,11(3), 687-705, 2003.

[28]T.Bag, S.K.Samanta,Fuzzy bounded linear operators, Fuzzy sets and Systems, 151, 513-547, 2005.

[29]Clementina Felbin, Finite dimensional fuzzy normed linear spaces,Fuzzy sets and Systems,48(2), 239-248, 1992.

[30]Keun Young Lee, Approximation properties in fuzzy normed spaces,282,115-130, 2016.

\vspace{0.5cm}

\end{document}